\documentclass[a4paper,10pt]{amsart}
\usepackage{amsmath,amssymb,amsthm,leftidx,cleveref,xcolor}

\renewcommand{\>}{\rangle}
\newcommand{\<}{\langle}

\newcommand{\kolomtwee}[2]{\left ( \begin{matrix} #1\cr #2 \end{matrix} \right)}

\newcommand{\CR}{\bar{\partial}}

\newcommand{\ev}{\operatorname{ev}}

\newcommand{\Crit}{\operatorname{Crit}}

\newcommand{\prt}{\operatorname{part}}
\newcommand{\field}{\operatorname{field}}
\newcommand{\inter}{\operatorname{inter}}
\newcommand{\contr}{\operatorname{contr}}

\newcommand{\diag}{\operatorname{diag}}

\newcommand{\del}{\partial}

\newcommand{\IC}{\operatorname{\mathbb{C}}}
\newcommand{\IZ}{\operatorname{\mathbb{Z}}}

\newcommand{\IR}{\operatorname{\mathbb{R}}}
\newcommand{\IN}{\operatorname{\mathbb{N}}}
\newcommand{\IH}{\operatorname{\mathbb{H}}}

\newcommand{\IT}{\operatorname{\mathbb{T}}}

\newcommand{\MM}{\operatorname{\mathcal{M}}}

\renewcommand{\AA}{\operatorname{\mathcal{A}}}
\newcommand{\PP}{\operatorname{\mathcal{P}}}

\newtheorem{theorem}{Theorem}[section]
\newtheorem{proposition}[theorem]{Proposition}

\newtheorem{lemma}[theorem]{Lemma}
\newtheorem{corollary}[theorem]{Corollary}
\newtheorem{remark}[theorem]{Remark}

\title{Cuplength estimates for\\ periodic solutions of\\ Hamiltonian particle-field systems}
\author{Oliver Fabert, Niek Lamoree}
\thanks{O. Fabert, VU Amsterdam, The Netherlands. Email: o.fabert@vu.nl}
\begin{document}
\maketitle

\begin{abstract}
We consider a natural class of time-periodic infinite-dimensional nonlinear Hamiltonian systems modelling the interaction of a classical mechanical system of particles with a scalar wave field. When the field is defined on a space torus $\IT^d=\IR^d/(2\pi\IZ)^d$ and the coordinates of the particles are constrained to a submanifold $Q\subset\IT^d$, we prove that the number of $T$-periodic solutions of the coupled Hamiltonian particle-field system is bounded from below by the $\IZ_2$-cuplength of the space $\Lambda^{\contr} Q$ of contractible loops in $Q$, provided that the square of the ratio $T/2\pi$ of time period $T$ and space period $X=2\pi$ is a Diophantine irrational number. The latter condition is necessary since for the infinite-dimensional version of Gromov-Floer compactness as well as for the $C^0$-bounds we need to deal with small divisors.  
\end{abstract}

\tableofcontents
\markboth{O. Fabert, N. Lamoree}{Cuplength estimates for particle-field systems} 

\section{A simple particle-field system}

As a motivation for the general class of Hamiltonian particle-field systems that we introduce below, we start with the following simple model of a classical particle-field system describing the interaction of a particle with its self-generated scalar wave field. After imposing periodicity conditions in space, we consider a (time-dependent) scalar field $\varphi(t,x)\in\IR$, $x\in\IT^d=\IR^d/(2\pi\IZ)^d$, as well as a particle whose center locus $q(t)$ is constrained to a closed submanifold $Q\subset\IT^d$. While the case $d>3$ only becomes relevant if we also allow for mechanical systems with more than one particle, we will ambiguously speak of a single particle. We assume that $q(t)$ and $\varphi(t,x)$ satisfy the following coupled system of differential equations,
\begin{eqnarray*}
 \nabla_t^2 q(t) &=& - \nabla V(t,q(t))-\nabla(\varphi*\rho)(t,q(t)),\\
 \del_t^2 \varphi(t,x) &=& \Delta\varphi(t,x)-\varphi(t,x)-\rho(x-q(t)).
\end{eqnarray*}
Here $\Delta=\del_{x_1}^2+\ldots+\del_{x_d}^2$ denotes the Laplacian for functions on $\IT^d$, while $\nabla$ denotes the gradient for functions on $Q\subset\IT^d$ as well as the Levi-Civita connection on the tangent bundle $TQ$ of $Q$ with respect to the induced canonical Riemannian metric. Furthermore, $$(\varphi*\rho)(t,q)\,=\,\int_{\IT^d} \varphi(t,x)\cdot\rho(q-x)\,dx$$ denotes convolution with respect to the space coordinate $x\in\IT^d$ with a fixed bump function $\rho\in C^\infty(\IT^d,\IR)$ which models the particle-field interaction and thereby takes into account the shape and density of the particle. We emphasize that all studied models for particle-field interaction such as in \cite{BaGa}, \cite{Kun}, \cite{Spo} use a smooth bump function $\rho$ modelling a realistic particle instead of the $\delta$-distribution modelling a point particle in order to avoid problems with singularities. Finally $V\in C^\infty(\IR/(T\IZ)\times\IT^d,\IR)$, $V_t(x)=V(t,x)$ denotes an external field on $(Q\subset) \IT^d$ which is explicitly assumed to be time-dependent with time period $T$.\\

Introducing a time-dependent momentum variable $p(t)\in T^*_{q(t)} Q$ and time-dependent momentum field $\pi(t,x)\in\IR$, $x\in\IT^d$, it is immediate to see that the above coupled system of differential equations is equivalent to the coupled system of differential equations
\begin{eqnarray}
 \del_t q = p ,& \nabla_t p= -\nabla V_t(q)-\nabla(\varphi*\rho)(q),\label{test}\\
 \del_t \varphi = \pi,& \del_t \pi = \Delta\varphi-\varphi-\rho(q-\cdot).\notag
\end{eqnarray}
Note that here the induced Riemannian metric on $Q$ is also used to canonically identify $T^*Q\cong TQ$.\\

It is well-known that the pair of equations $$\del_t q = p,\quad \nabla_t p= -\nabla V_t(q)$$ for $(q(t),p(t))$ is Hamiltonian for the Hamiltonian function $$H_{\prt}(q,p)=\frac{|p|^2}{2}+V_t(q)$$ on the phase space $T^*Q$, which is a finite-dimensional symplectic manifold. On the other hand, it is also a classical fact that the pair of equations $$\del_t \varphi = \pi,\quad \del_t \pi = \Delta\varphi-\varphi$$ for $(\varphi(t,x),\pi(t,x))$ is Hamiltonian for the Hamiltonian function 
\begin{equation*}
H_{\field}(\varphi,\pi)=\int_{\IT^d} \left(\frac{|\pi(x)|^2}{2}+\frac{|\nabla \varphi(x)|^2}{2}+\frac{|\varphi(x)|^2}{2}\right)\,dx
\end{equation*} 
on the phase space $L^2(\IT^d,\IR)\oplus L^2(\IT^d,\IR)$. Note that $L^2(\IT^d,\IR)\oplus L^2(\IT^d,\IR)$ is an infinite-dimensional symplectic Hilbert space, where the symplectic form is given by the $L^2$ inner product as well as the complex structure $J\cdot(\varphi,\pi)=(\pi,-\varphi)$. Taking the particle-field interactions into account, we stress that the coupled system of differential equations is now a Hamiltonian system on the infinite-dimensional symplectic manifold $T^*Q\times L^2(\IT^d,\IR)\oplus L^2(\IT^d,\IR)$ for the Hamiltonian function $$H=H_{\prt}+H_{\field}+H_{\inter}$$ with $$H_{\inter}(q, \varphi,\pi)=(\varphi*\rho)(q)=\int_{\IT^d}\varphi(x)\rho(q-x)\,dx$$ denoting the Hamiltonian function modelling the interaction between particle and field.\\

While this set-up already seems promising, there is an alternative Hamiltonian structure for the wave equation which is more symmetric and turns out to be more suitable for the analysis, see \cite{Kuk1}, \cite{Kuk2}. From now on we will consider the modified coupled system 
\begin{eqnarray*}
 \del_t q = p,& \nabla_t p= -\nabla V_t(q)-\nabla(\varphi*\rho)(q)\\
 \del_t \varphi = B\pi,& \del_t \pi = - B\varphi-B^{-1}\rho(q-\cdot).
\end{eqnarray*}
with $B=\sqrt{1-\Delta}$. While this modified system is still Hamiltonian, the new phase space is $T^*Q\times\IH$ with $\IH=H^{\frac{1}{2}}(\IT^d,\IR)\oplus H^{\frac{1}{2}}(\IT^d,\IR)$, where the symplectic form on $\IH$ is again given by the standard complex structure $J\cdot(\varphi,\pi)=(-\pi,\varphi)$ and the standard inner product $\<\cdot,\cdot\>$ on $H^{\frac{1}{2}}(\IT^d,\IR)$ given by $$\<f,g\>=\int_{\IT^d} f(x) (Bg)(x)\,dx.$$ While it can be observed that the interaction Hamiltonian can remain unchanged, $$H_{\inter}(q,\varphi,\pi)=(\varphi*\rho)(q)=\<\varphi,B^{-1}\rho(q-\cdot)\>,$$
the field Hamiltonian is now changed to the more symmetric form 
$$H_{\field}(\varphi,\pi)=\frac{1}{2}\<\varphi,B\varphi\>+\frac{1}{2}\<\pi,B\pi\>.$$

\section{Hamiltonian particle-field systems}

For more details and background on Hamiltonian PDEs we refer to \cite{Kuk1}, see also \cite{AM} and \cite{Kuk2}. We start by recalling that $\IH=H^{\frac{1}{2}}(\IT^d,\IR)\oplus H^{\frac{1}{2}}(\IT^d,\IR)$ is a separable Hilbert space which is equipped with a (strongly) symplectic form $\omega_{\IH}=\langle J_{\IH}\cdot,\cdot\rangle_{\IH}:\IH\times\IH\to\IR$ given by the standard complex structure $J_{\IH}\cdot(\varphi,\pi)=(-\pi,\varphi)$ and the standard inner product $\<\cdot,\cdot\>_{\IH}$ on $H^{\frac{1}{2}}(\IT^d,\IR)\oplus H^{\frac{1}{2}}(\IT^d,\IR)$. Here being (strongly) symplectic means that $\omega$ is anti-symmetric and defines an isomorphism between $\IH$ and its dual $\IH^*$. There exists a complete basis $(e_n^\pm)_{n\in\IZ^d}$ of $\IH$ with $\omega_{\IH}(e_n^+,e_m^-)=\delta_{n,m}$, $J_{\IH}e_n^\pm:=\pm e_n^\mp$ of the form $e_n^+=(\xi_n,0)$, $e_n^-=(0,\xi_n)$, where $(\xi_n)_{n\in\IZ^d}$ is a suitably normalized complete basis of $H^{\frac{1}{2}}(\IT^d,\IR)$ in terms of sine and cosine functions. Furthermore, after identifying $\IH$ with the subspace of $\IH\otimes\IC$ on which $J_{\IH}=i$, note that there is a complete unitary basis $(z_n)_{n\in\IZ^d}$ which further allows us to identify $\IH$ with the Hilbert space $\ell^2(\IZ^d,\IC)$. \\

The symplectic Hilbert space $\IH=H^{\frac{1}{2}}(\IT^d,\IR)\oplus H^{\frac{1}{2}}(\IT^d,\IR)$ naturally comes equipped with a symplectic Hilbert scale $(\IH_h)_{h\in\IR}$ with $$\IH_h=H^{\frac{1}{2}+h}(\IT^d,\IR)\oplus H^{\frac{1}{2}+h}(\IT^d,\IR)$$ in the sense that the inclusion $\IH_h\subset\IH_i$ is compact and dense if $h>i$, and $\omega_{\IH}$ defines an isomorphism between the Hilbert spaces $\IH_h$ and $\IH_{-h}^*$. Furthermore we define $$\IH_\infty =\bigcap_{h\in\IR}\IH_h,\,\, \IH_{-\infty} = \bigcup_{h\in\IR}\IH_h.$$ 
In particular, we stress that $\omega_{\IH}:\IH_\infty\to\IH_\infty^*$ is only injective and hence $\omega_{\IH}$ only defines a weakly symplectic form on the Frechet space $\IH_\infty=C^{\infty}(\IT^d,\IR)\oplus C^{\infty}(\IT^d,\IR)$. \\

The above setup immediately generalizes from $\IH$ to $\widetilde{M}:=M\times\IH$, where $(M,\omega_M)=(T^*Q,d\lambda_M)$ is the cotangent bundle of the closed manifold $Q$ with its canonical symplectic form given by the Liouville one-form $\lambda_M$. We assume that $\widetilde{M}$ carries the product symplectic form $\omega=\pi_M^*\omega_M+\pi_{\IH}^*\omega_{\IH}$, where $\pi_M$, $\pi_{\IH}$ denote the projection onto the first or second factor of $\widetilde{M}$, respectively. Note that the Riemannian metric on the submanifold $Q$, obtained from the embedding into $\IT^d$, leads to a natural choice of Riemannian metric $\langle\cdot,\cdot\rangle_M$ on $T^*Q$, defined using the Levi-Civita connection on $T^*M\cong TM$, with corresponding $\omega_M$-compatible almost complex structure $J_M$, see \cite{C} for details. It follows that $\omega=\langle J\cdot,\cdot\rangle$ for the canonical Riemannian metric $\langle\cdot,\cdot\rangle=\pi_M^*\langle\cdot,\cdot\rangle_M+\pi_{\IH}^*\langle\cdot,\cdot\rangle_{\IH}$ and the canonical almost complex structure $J=\diag(J_M,J_{\IH})$ on $\widetilde{M}=T^*Q\times\IH$. Furthermore the generalization of the scale structure is given by $\widetilde{M}_h:=M\times\IH_h$ for $h\in\IR$.\\

As in \cite{F}, \cite{FL} we consider a class of time-dependent Hamiltonians $H_t$ of the form $$H_t(u)=H^A(u)+F_t(u)\,\,\textrm{with}\,\,H^A(u)=\frac{1}{2}\langle \pi_{\IH}u,A \pi_{\IH}u\rangle,\,\,u=(q,p,\varphi,\pi),$$
where $A$ is a differential operator on $\IH$ such that $(e_n^\pm)_{n\in\IZ^d}$ is a basis of eigenvectors for $A$ with real eigenvalues, and $F_t: \widetilde{M}\to\IR$ is time-periodic with period $T$ and smoothly depending on the time $t\in\IR$. Since in this paper we are particularly interested in particle-field systems, we restrict the general class of Hamiltonians to the case where \begin{eqnarray*}
H^A(u)&=&H_{\field}(\varphi,\pi)\,\,\textrm{with}\,\, H_{\field}(\varphi,\pi)=\frac{1}{2}\<\varphi,B\varphi\>+\frac{1}{2}\<\pi,B\pi\>\,\,\textrm{as before},\\ F_t(u)&=&F_{\prt,t}(q,p)+F_{\inter,t}(q,\varphi,\pi).
\end{eqnarray*}
Note that in this case $A=\diag(B, B)$ is of order $1$ and the eigenvalue corresponding to the eigenfunction $e_n^{\pm}$ is $\sqrt{n^2+1}$ with $n^2=n_1^2+\ldots+n_d^2$ for all $n=(n_1,\ldots,n_d)\in\IZ^d$. As in \cite{C} we assume that the particle Hamiltonian $F_{\prt,t}\in C^{\infty}(T^*Q,\IR)$ is asymptotically quadratic with respect to the momentum coordinates $p$ in the sense that \begin{itemize}
\item[(F1)] $dF_{\prt,t}(q,p)\cdot p\frac{\del}{\del p} - F_{\prt,t}(q,p)\geq c_0 |p|^2 - c_1$, for some constants $c_0>0$ and $c_1\geq 0$,
\item[(F2)] $\displaystyle \left|\frac{\del^2 F_{\prt,t}}{\del p_i\del p_j}(q,p)\right|,\,\, \left|\frac{\del^2 F_{\prt,t}}{\del p_i\del q_j}(q,p)\right|< c_2$ for some constant $c_2\geq 0$,
\end{itemize} 
while for the interaction Hamiltonian $F_{\inter,t}(q,\varphi,\pi)$ we require that 
\begin{itemize}
    \item[(F3)] $F_{\inter,t}(q,\varphi,\pi)=f_t((\varphi*\rho)(q),(\pi*\rho)(q))$ with $f_t\in C^{\infty}(\IR^2,\IR)$, $f_{t+T}=f_t$ having bounded first derivatives. 
\end{itemize}
As in \cite{C} $(q_i,p_i)_i$ are coordinates on $T^*Q$ induced by geodesic normal coordinates $(q_i)$ on $Q$. Finally we assume without loss of generality that all frequencies are present in $\rho=\sum_{n\in\IZ^d}\hat{\rho}(n)z_n$ in the sense that $\hat{\rho}(n)\neq 0$ for all $n\in\IZ^d$. \\

In analogy with our work in \cite{F} and \cite{FL}, it is the goal of this paper to prove the existence of $T$-periodic solutions of $$\partial_t u\,=\,X^H_t(u)\,=\,JA\pi_{\IH}u\;+\;J\nabla F_t(u).$$ Before we can state the main result, we recall that an irrational number $\sigma$ is called Diophantine if there exists $c>0$ and $r>0$ such that $$\inf_{m\in\IZ}\Big|\sigma-\frac{m}{n}\Big|\,\geq\, c\cdot n^{-r}\,\,\textrm{for all}\,\, n\in\IN.$$ 
We stress that the set of Diophantine numbers has full measure. Furthermore the $\IZ_2$-cuplength $c\ell_{\IZ_2}(\Lambda^{\contr}Q)$ of the space $\Lambda^{\contr}Q=C^0_{\contr}(\IR/(T\IZ),Q)$ of contractible loops in $Q$ is defined as 
\begin{eqnarray*}
c\ell_{\IZ_2}(\Lambda^{\contr}Q)=\sup\{N+1:\exists\theta_1,\ldots\theta_N\in H^{*\neq 0}(\Lambda^{\contr}Q,\IZ_2): \theta_1\cup\ldots\cup\theta_N\neq 0\}.
\end{eqnarray*}

\begin{theorem}\label{periodic-orbits}
Assume that the squared ratio $T^2/(2\pi)^2$ of time and space period is a Diophantine irrational number, $F_{\prt,t}$ satisfies (F1), (F2), and $F_{\inter,t}$ satisfies (F3). Then the number of $T$-periodic solutions $$u=(q,p,\varphi,\pi):\IR/(T\IZ)\to T^*Q\times H^{\frac{1}{2}}(\IT^d,\IR)\oplus H^{\frac{1}{2}}(\IT^d,\IR)$$ of the coupled system of Hamiltonian equations 
\begin{eqnarray*}
 \del_t q &=& \nabla_p F_{\prt,t},\\ 
 \nabla_t p &=& -\nabla_q F_{\prt,t} - \del_1 f_t\cdot\nabla(\varphi*\rho)(q) - \del_2 f_t\cdot\nabla(\pi*\rho)(q),\\
 \del_t \varphi &=& B\pi + \del_2 f_t\cdot B^{-1}\rho(q-\cdot),\\ 
 \del_t \pi &=& - B\varphi - \del_1 f_t\cdot B^{-1}\rho(q-\cdot)
\end{eqnarray*}
is bounded from below by $c\ell_{\IZ_2}(\Lambda^{\contr} Q)$; in particular, it is infinite if $\pi_1(Q)$ is finite. Furthermore $u$ has image in the weakly symplectic manifold $T^*Q\times C^\infty(\IT^d,\IR)\oplus C^\infty(\IT^d,\IR)$.
\end{theorem}

Here  $\nabla_qF_{\prt,t},\;\nabla_pF_{\prt,t}$ denote the components of the gradient $\nabla F_{\prt,t}$ with respect to the splitting $TT^*Q\cong TQ\oplus T^*Q$ obtained using the Levi-Civita connection for the canonical Riemannian metric on $T^*Q$, and $\del_{1,2}f_t$ denote the first derivatives of $f_t$ with respect to the first and second coordinate. Further recall that $B=\sqrt{1-\Delta}$. \\ 

We emphasize that \Cref{periodic-orbits} generalizes the celebrated cuplength result for Hamiltonian systems on cotangent bundles in \cite{C} to the case of Hamiltonian particle-field systems: When $F_{\inter,t}=0$, then our result is equivalent with \cite[Main Theorem 1.1]{C}, since the linear wave equation admits only the trivial solution in the case when $T^2/(2\pi)^2$ is irrational. Apart from the fact that we now consider an infinite-dimensional Hamiltonian system with a densely defined Hamiltonian function and a small divisor problem, it is a nontrivial observation that we can still establish $C^0$-bounds. For the latter the trick is to replace the original interaction Hamiltonian $F_{\inter,t}$ by a slightly modified version $\bar{F}_{\inter,t}$ such that the resulting modified particle-field Hamiltonian system however still has the same set of periodic orbits. This is the content of section \ref{properties}. As for the Gromov-Floer compactness theorem in infinite dimensions, we need to deal with small divisors and crucially use that the bump function $\rho$ is smooth. \\   

\begin{remark} We expect that \Cref{periodic-orbits} can be generalized in the following directions: 
\begin{enumerate}
\item By inspecting the proof given below it becomes apparent that the statement of \Cref{periodic-orbits} continues to hold true for interaction Hamiltonians $F_{\inter,t}$ which additionally depend on the $p$-component and the condition (F3) is generalized to (F3'):  $F_{\inter,t}(q,p,\varphi,\pi)=f_t(q,p,(\varphi*\rho)(q),(\pi*\rho)(q))$ with $f_t\in C^{\infty}(T^*Q\times\IR^2,\IR)$, $f_{t+T}=f_t$ having bounded first derivatives with respect to the third and fourth entry, and $\nabla_p F_{\inter,t}\equiv 0$ for $|p|>R$ for some $R>0$. 
Since this generalization will follow from standard arguments in Floer theory and we want to focus on the new challenges that appear when applying ideas of Floer theory to particle-field Hamiltonians, we choose to work with the more restrictive condition (F3).
\item Generalizing from cotangent bundle to general symplectic manifolds and from scalar fields to fields that are sections in arbitrary vector bundles, we claim that ideas entering the proof of \Cref{periodic-orbits} can be used to prove existence results of periodic orbits for suitably generalized classes of Hamiltonian functions in the following generalized geometric framework: Consider a symplectic manifold $(B,\omega_B)$ with a foliation by Lagrangian submanifolds, which contains $(M,\omega_M)$ as a symplectic submanifold, as well as a symplectic vector bundle $E\to B$ over $B=(B,\omega_B)$. Let $\IH=(\IH,\omega_{\IH})$ denote a symplectic Hilbert space of sections in this bundle which are constant along the leaves, where the symplectic bilinear form $\omega_{\IH}$ on $\IH$ is defined using the symplectic structures on the fibres. Now consider time-periodic Hamiltonians $H_t=H^A+F_t:M\times\IH\to\IR$ with $F_t(u_M,u_{\IH})=f_t(u_M,u_{\IH}^\rho(u_M))$, where $u_{\IH}\mapsto u_{\IH}^\rho$ denotes a smoothing operator $\IH_{s-h}\to\IH_s$ for all $s\in\IR$. Note that setting $M=T^*Q\subset T^*T^d=B$, $E=B\times\IC$, and viewing  $\IH=H^{\frac{1}{2}}(T^d,\IC)$ as a space of sections in the trivial bundle that are constant along leaves of the canonical Lagrangian foliation on $T^*T^d$ given by the cotangent fibres, we see that this is indeed a generalization of the class of Hamiltonian particle-field systems considered above. The case when $(M,\omega_M)$ is closed and aspherical is treated in \cite{FL}. 
\end{enumerate}
\end{remark}

Since the particle-field Hamiltonian of the simple particle-field system that we consider above satisfies (F1), (F2), and (F3), we have the following consequence.  
\begin{corollary}\label{classical}
The number of $T$-periodic solutions $u=(q,\varphi):\IR/(T\IZ)\to Q\times C^\infty(\IT^d,\IR)$ of the simple particle-field system
\begin{eqnarray*}
 \nabla_t^2 q(t) &=& - \nabla V(t,q(t))-\nabla(\varphi*\rho)(t,q(t)),\\
 \del_t^2 \varphi(t,x) &=& \Delta\varphi(t,x)-\varphi(t,x)-\rho(x-q(t)).
\end{eqnarray*}
is bounded from below by $c\ell_{\IZ_2}(\Lambda^{\contr} Q)$, provided that the squared ratio $T^2/(2\pi)^2$ of time and space period is a Diophantine irrational number. In particular, it is infinite if $\pi_1(Q)$ is finite. \end{corollary}

We claim that this result is already interesting in the case when there is no exterior time-dependent potential $V$: While in the case without interaction ($\rho=0$) one only can prove the existence of at least one closed geodesic as all periodic orbits could be iterates of the same geodesic, in the case when $\rho\neq 0$ it can be checked that iterates of $t\mapsto q(t)$ no longer lead to solutions $t\mapsto (q(t),\varphi(t))$ of \begin{eqnarray*}
 \nabla_t^2 q(t) &=& - \nabla(\varphi*\rho)(t,q(t)),\\
 \del_t^2 \varphi(t,x) &=& \Delta\varphi(t,x)-\varphi(t,x)-\rho(x-q(t)).
\end{eqnarray*}
More precisely, it follows from the proof of \Cref{bounded} that a solution $t\mapsto (q(t),\varphi(t))$ could at most lead to a finite number of solutions by taking iterates of $t\mapsto q(t)$, since $\nabla(\varphi*\rho)(t,q(t))$ is bounded uniformly with respect to $(q,\varphi)$. In reality this result for instance establishes the existence of infinitely many periodic orbits of a charged particle on a $2$-sphere in an electric field when the interaction of the particle with its self-generated electric field is no longer neglected. In order to illustrate the nontriviality of the problem of establishing $C^0$-bounds, note that we currently do not know how to extent our result to the case when the magnetic Lorentz force would be included. Furthermore our result indeed crucially relies on the fact that the underlying space $\IT^d=\IR^d/(2\pi\IZ)^d$ is periodic: Informally speaking, if no space periodicity was assumed, then the energy that was transferred from the mechanical system to the field would disappear towards infinity and hence would have no chance to come back to the particle, see \cite{Kun} for a negative result for a closely related problem. \\

While the Hamiltonian equations for the particle and the Hamiltonian equations for the field are coupled via the interaction terms and hence cannot be solved individually, let us nonetheless look at both separately:
\begin{enumerate}
    \item Assuming first that $\varphi(t,x)$ is given and smooth and time-periodic with period $T$, it was first shown by Benci in \cite{Ben} that the number of smooth $T$-periodic solutions of the Hamiltonian equation \eqref{test} describing the motion of the particle is bounded from below by the $\IZ_2$-cuplength of $\Lambda^{\contr} Q$. Here it is crucially used that this special particle Hamiltonian function is obtained from a Lagrangian function using the Legendre transform. 
    \item Now assume that $q(t)$ is given and smooth and time-periodic with time period $T$. Since we are looking for wave functions $\varphi(t,x)$ which are time-periodic with period $T$ as well, we assume that $\varphi(t,x)$ and $f(t,x)=\rho(q(t)-x)$ can be written as Fourier series. It follows that the corresponding Fourier coefficients $\hat{\varphi}(m,n)$, $\hat{f}(m,n)$ must satisfy the equation $$\left(\left(\frac{T}{2\pi}\right)^2-\frac{m^2}{n^2+1}\right)\hat{\varphi}(m,n)=\frac{T^2}{(2\pi)^2 (n^2+1)}\hat{f}(m,n).$$ When $T^2/(2\pi)^2$ is irrational, it follows that the first factor is never zero; however there exists a subsequence of tuples $(m,n)$ for which it converges to zero. When $T^2/(2\pi)^2$ is Diophantine, then this however only happens with at most polynomial speed, while the Fourier coefficients $\hat{f}(m,n)$ converge to zero with exponential speed due to the smoothness of $f(t,x)$. This is sufficient to ensure that also the Fourier coefficients $\hat{\varphi}(m,n)$ converge to zero with exponential speed, which proves that also the Hamiltonian equations for the field have a smooth $T$-periodic solution, provided that $T^2/(2\pi)^2$ is a  Diophantine irrational number. 
\end{enumerate}

While the Hamiltonian function for the simple particle-field system is obtained from a Lagrangian function using the Legendre transform, it is actually not known to us whether the existence of periodic solutions for this special Lagrangian particle-field system has been established before using a generalization of Benci's method.   

\section{Properties of the particle-field Hamiltonian}\label{properties}

In order to prove \Cref{periodic-orbits}, we essentially combine our work in \cite{F} and \cite{FL} on pseudoholomorphic curves for infinite-dimensional Hamiltonian systems with the celebrated paper \cite{C} on pseudoholomorphic curves in cotangent bundles in order to prove an existence result for Floer curves in $T^*Q\times\IH$.\\

Apart from the fact that the underlying phase space $\widetilde{M}=T^*Q\times\IH$ is infinite-dimensional, as in \cite{F}, \cite{FL} we first need to deal with the fact that the field Hamiltonian $$H^A(u)=\frac{1}{2}\langle u_{\IH}, Au_{\IH}\rangle=\frac{1}{2}\<\varphi,B\varphi\>+\frac{1}{2}\<\pi,B\pi\>\,\,\textrm{with}\,\,u_{\IH}=\pi_{\IH}u=(\varphi,\pi)$$ is only defined on the dense subspace $\IH_{\frac{1}{2}}=H^1(\IT^d,\IR)\oplus H^1(\IT^d,\IR)$ of the symplectic Hilbert space $\IH$. \\

As in \cite[Prop. 2.1]{F}, \cite[Section 1]{FL}, one finds that this problem can be resolved by only working with the Hamiltonian flow $\phi^A_t$ of $H^A$, defined by $$\del_t|_{t=0}\phi^A_t(u) = X^A(u) = J A u_{\IH},$$ since it has much better properties: While it is trivial on $T^*Q$, for every fixed time $t\in\IR$ it extends to a unitary linear map on $\IH$, since $$\phi^A_t\cdot z_n = \exp(\sqrt{n^2+1}\cdot it)\cdot z_n.$$ In particular, for every chosen time period $T$ it follows that the time-$T$ map $\phi^A_T$ is a smooth symplectomorphism of $(\widetilde{M},\omega)$ which furthermore preserves the complex structure $J$ on $\widetilde{M}$. While the Hamiltonian flow $\phi^A_t$ is smooth with respect to the time coordinate $t$ on $\IH_{\infty}=C^\infty(\IT^d,\IR)\oplus C^\infty(\IT^d,\IR)$, we stress that on $\IH=H^{\frac{1}{2}}(\IT^d,\IR)\oplus H^{\frac{1}{2}}(\IT^d,\IR)$ it is only continuous, see also \cite[Prop. 2.5]{F}. \\

Defining for each $u:\IR\to T^*Q\times\IH$ a map $\bar{u}:\IR\to T^*Q\times\IH$ via  $\bar{u}(t)=\phi^A_{-t}\cdot u(t)$, it follows as in \cite[Prop. 2.5]{F} that $u$ solves $\partial_t u=JA\pi_{\IH}u+J\nabla F_t(u)$ with $u(t+T)=u(t)$ if and only if $\bar{u}$ solves $\partial_t\bar{u}=J\nabla G_t(\bar{u})$, $\bar{u}(t+T)=\phi^A_{-T}\cdot\bar{u}(t)$, where we define $G_t=F_{\prt,t}+G_{\inter,t}$ with $G_{\inter,t}=F_{\inter,t}\circ\phi^A_{-t}$. Writing $u_{\IH}*\rho=(\varphi*\rho,\pi*\rho)$, note that since $G_{\inter,t}(q,\varphi,\pi)=f_t((\phi^A_{-t}u_{\IH}*\rho)(q))=f_t(\phi^A_{-t}(u_{\IH}*\rho)(q))$, it follows that $G_{\inter,t}$ is still smooth with respect to the time coordinate $t$, using that $u_{\IH}*\rho\in\IH_{\infty}$. \\

Denote by $\PP(\phi^A_T,G)$ the set of $\phi^A_T$-periodic solutions $\bar{u}$ of $\del_t\bar{u}=J\nabla G_t(\bar{u})$ and for given $R>0$ let $\chi_R:\IR\to [0,1]$ denote a smooth cut-off function with $\chi_R(r)=1$ for $r\leq R$ and $\chi_R(r)=0$ for $r\geq R+1$.

\begin{lemma}\label{bounded}
There exists some $R_0>0$ such that $\displaystyle|\pi_{\IH}\bar{u}(t)|_1\leq R_0$ for all $\bar{u}\in\PP(\phi^A_T,G)$ and $t\in\IR$. In particular, we find $R_1>0$ such that $\PP(\phi^A_T,G)=\PP(\phi^A_T,\tilde{G})$ with $\tilde{G}_t=F_{\prt,t}+\tilde{G}_{\inter,t}$, $\tilde{G}_{\inter,t}=\tilde{F}_{\inter,t}\circ\phi^A_{-t}$ and $$\tilde{F}_{\inter,t}(u)=\chi_{R_1}(|u_{\IH}*\rho|_{C^3})\cdot F_{\inter,t}(u).$$
\end{lemma}

\begin{proof} As in the proof of \cite[Theorem 9.2]{FL}, see also \cite[Lemma 9.1]{FL}, we show that, when the $\IH_1$-norm $|u_{\IH}|_1$ of $u_{\IH}=\pi_{\IH}u$ is too large, the map $\phi^A_T$ moves the point $u_{\IH}=(\varphi,\pi)$ further than $\phi_T^G$ can move points, so that $u_{\IH}$ cannot be a fixed point of the time-$T$ flow $\phi_T^H=\phi^A_T\circ\phi_T^G$ of the Hamiltonian vector field of $H_t=H^A+F_t$. As a consequence, multiplying with a cut-off function with cut-off region outside of this region will not alter the set of periodic points. Note that, since the particle Hamiltonian $F_{\prt,t}$ does not depend on $u_{\IH}$, it suffices to take only the interaction Hamiltonian $F_{\inter,t}$ into account.\par
Since $T^2/(2\pi)^2$ is assumed to be a Diophantine irrational number, it follows that there exists $c>0$ and $r>2$ such that $$\inf_{m\in\IN}\left|\left(\frac{T}{2\pi}\right)^2-\frac{m^2}{n^2+1}\right|\geq c\cdot (n^2+1)^{-r}.$$ Since $$\left(\frac{T}{2\pi}\right)^2-\frac{m^2}{n^2+1}=\left(\frac{T}{2\pi}-\frac{m}{\sqrt{n^2+1}}\right)\cdot\left(\frac{T}{2\pi}+\frac{m}{\sqrt{n^2+1}}\right),$$ and the second factor is approximately equal to $2\cdot T/(2\pi)$ whenever the first factor is close to zero, it follows that there also exists some $c'>0$ and $r'=r-\frac{1}{2}$ such that $$\inf_{m\in\IN}\left|T\cdot\sqrt{n^2+1}-m\cdot 2\pi\right|=2\pi\cdot\sqrt{n^2+1}\cdot\inf_{m\in\IN}\left|\frac{T}{2\pi}-\frac{m}{\sqrt{n^2+1}}\right|> c'\cdot (n^2+1)^{-r'}.$$ Together with $\phi^A_T\cdot z_n=\exp(i\cdot T\cdot\sqrt{n^2+1})\cdot z_n$, it follows with the small angle approximation that $$|\phi^A_T\cdot z_n - z_n|> c'\cdot (n^2+1)^{-r'}.$$ It follows that there exists $c''>0$ and $h_0>0$ such that $$|\phi^A_T(u_{\IH})-u_{\IH}|_h>c''\cdot |u_{\IH}|_{h-h_0} \,\,\textrm{for every}\,\, h\in\IR,$$ where $|\cdot|_h$ denotes the Hilbert space norm on $\IH_h=H^{h+\frac{1}{2}}(\IT^d,\IR)$. On the other hand, since
$F_{\inter,t}(q,u_{\IH})=f_t((u_{\IH}*\rho)(q))$ satisfies (F3) and $\rho\in C^{\infty}(\IT^d,\IR)$, it follows that the $\IH$-component $$\nabla^{\IH}F_{\inter,t}(q,u_{\IH})=B^{-1}\rho(q-\cdot)\cdot(\partial_1f_t((u_{\IH}*\rho)(q)),\partial_2f_t((u_{\IH}*\rho)(q)))$$ of the gradient of $F_{\inter,t}$
%$G_{\inter,t}(q,u_{\IH})=f_t((u_{\IH}*\phi^A_{-t}\rho)(q))$ satisfies (F3) and hence the $\IH$-component  $$\nabla^{\IH}G_{\inter,t}(q,u_{\IH})=Df_t((u_{\IH}*\phi^A_{-t}\rho)(q))\cdot(B^{-1}\phi^A_{-t}\rho)(q-\cdot)$$ of the gradient of $G_{\inter,t}$ 
is bounded with respect to the $\IH_h$-norm for every $h\in\IR$. It follows that the same holds for the ${\IH}$-component of the gradient of $G_{\inter,t}$ since $\phi^A_t$ is a unitary map, and so for every $h\in\IR$ there exists $c_h'''>0$ with $$|\pi_{\IH}\phi^{G_{\inter}}_T(u)-\pi_{\IH}u|_h\leq c'''_h.$$ Choosing $h\in\IR$ such that $h-h_0=1$, we find that $|\pi_{\IH}\bar{u}(t)|_1\leq R_0$ for every $\bar{u}\in\PP(\phi^A_T,G)$ and $t\in\IR$ with $R_0:=c'''_h/c''$. The second statement follows from the fact that there exists $c''''>0$ such that $|u_{\IH}*\rho|_{C^3}\leq c''''\cdot |u_{\IH}|_1$ for every $u\in T^*Q\times\IH$. Hence, when we define $\tilde{G}_{\inter,t}(u):=\chi_{R_1}(|\pi_{\IH}u*\rho|_{C^3})\cdot G_{\inter,t}(u)$  with $R_1=c''''R_0$, and $\tilde{G}_t:=F_{\prt,t}+\tilde{G}_{\inter,t}$, then we find that $\PP(\phi^A_T,\tilde{G})=\PP(\phi^A_T,G)$. \end{proof}
Note that we can replace the $C^3$-norm by the $C^\kappa$-norm for any $\kappa$. This would lead to Floer curves of higher regularity in \Cref{compactness}. However, since we only require the Floer curves to be of class $C^1$, the $C^3$-norm suffices. It follows that instead of considering the original Hamiltonian $G_t=F_{\prt,t}+G_{\inter,t}$ from now on we can work with the Hamiltonian $\tilde{G}_t=F_{\prt,t}+\tilde{G}_{\inter,t}: T^*Q\times\IH\to\IR$ which has bounded support on $\IH$ in a weak sense. Note that in the language of \cite[Definition 3.3]{FL} we would call $G_{\inter,t}$ weakly $A$-admissible, while $\tilde{G}_{\inter,t}$ would be called $A$-admissible. In order to illustrate the benefit that we gained from passing from $G_t$ to $\tilde{G}_t$, we prove the following 

\begin{lemma} \label{G1G2}
$\tilde{F}_{\inter,t}$ and hence also $\tilde{G}_{\inter,t}$ has finite $C^3$-norm.
\end{lemma}

\begin{proof} 
We establish the $C^3$-bound for $\tilde{F}_{\inter,t}(u)=\chi_{R_1}(|u_{\IH}*\rho|_{C^3})\cdot f_t((u_{\IH}*\rho)(q))$; the bound for $\tilde{G}_{\inter,t}$ follows from the fact that $\phi^A_t$ is a unitary map. Assuming boundedness of the $C^3$-norm of $u_{\IH}*\rho$, it is clear that the map $(q,u_{\IH})\mapsto(u_{\IH}*\rho)(q)$ has bounded $q$-derivatives up to order $3$. Since the $\IH$-gradient is given by $(B^{-1}\rho(q-\cdot),B^{-1}\rho(q-\cdot))$ and in particular does not depend on $u_{\IH}$, boundedness of all derivatives up to order $3$ follows. Since $f_t$ and the cut-off function are smooth and $\tilde{F}_{\inter,t}$ has bounded support in $\{u\vert |u_{\IH}*\rho|_{C^3}\leq R_1+1\}$, it follows that derivatives of $\tilde{F}_{\inter,t}$ up to order $3$ are bounded. \end{proof}

Set $\phi:=\phi^A_T$. The symplectic action $\AA_{\phi}^{\tilde{G}}$ of a $\phi$-periodic solution $\bar{u}$ of $\del_t\bar{u}=J\nabla\tilde{G}_t(\bar{u})$ is defined as the symplectic action of the corresponding $T$-periodic solution $u=(u_M,u_{\IH})$ of $\del_t u=J\nabla\tilde{H}_t(u)$ with $\tilde{H}_t=H^A+\tilde{F}_t$ given by $$\AA^{\tilde{H}}(u)=\int\tilde{u}^*\omega-\int_0^T \tilde{H}_t(u)\,dt\,\textrm{with}\,\,\tilde{u}:D^2\to T^*Q\times\IH,\,\tilde{u}(e^{2\pi it})=u(t).$$ Note that this is equal to 
$$\int_0^T\left(\lambda_M(\del_t u_M)-\tilde{F}_t(u)\right)\;dt\,+\,\int_0^T\left(\<\pi,\del_t\varphi\>-H_{\field}(\varphi,\pi)\right)\;dt.$$ Denote by $\PP_{\leq a}(\phi,\tilde{G})$ the set of $\phi$-periodic orbits with symplectic action less than or equal to $a\in\IR$. 

\begin{lemma}\label{p-cutoff}
For every $a\in\IR$ there exists $R_2>0$ such that $\PP_{\leq a}(\phi,\bar{G})=\PP_{\leq a}(\phi,\tilde{G})$ with $\bar{F}_{\inter,t}(q,p,u_{\IH})=\chi_{R_2}(\ln|p|)\cdot\tilde{F}_{\inter,t}(q,u_{\IH})$ and $\bar{G}_{\inter,t}=\bar{F}_{\inter,t}\circ\phi^A_{-t}$, $\bar{G}_t=F_{\prt,t}+\bar{G}_{\inter,t}$.
\end{lemma}

\begin{proof} Let $\bar{u}$ denote a $\phi$-periodic orbit in $\PP_{\leq a}(\phi,\bar{G}_t)$ with corresponding $T$-periodic orbit $u=(q,p,\varphi,\pi)$ of $\bar{H}_t$. Since the $\IH_1$-norm of $(\varphi,\pi)$ is bounded by \Cref{bounded}, it follows that there exists $a_0\in\IR$ such that  $$\left|\int_0^T\left(\<\pi,\del_t\varphi\>-H_{\field}(\varphi,\pi)\right)\;dt\right|\leq a_0.$$ For this it suffices to observe that $\<\pi,\del_t\varphi\>-H_{\field}(\varphi,\pi)$ is given by $$\frac{1}{2}\<\pi,B\pi+2\del_2 f_t\cdot B^{-1}\rho(q-\cdot)\>-\frac{1}{2}\<\varphi,B\varphi\>.$$ With this we can compute as in \cite[Lemma 5.3]{C} 
\begin{eqnarray*}
a+a_0 &\geq& \int_0^T\left(\lambda_M(\del_t u_M)-\bar{F}_t(u)\right)\;dt\\ &=& \int_0^T \left(\omega_M(p\frac{\del}{\del p},J\nabla\bar{F}_t(u))-\bar{F}_t(u)\right)\;dt\\ &=& \int_0^T \left(d\bar{F}_t(u)\cdot p\frac{\del}{\del p}-\bar{F}_t(u)\right)\;dt\\ &=& \int_0^T \left(dF_{\prt,t}(u_M)\cdot p\frac{\del}{\del p}-F_{\prt,t}(u_M)+d\bar{F}_{\inter,t}(u)\cdot p\frac{\del}{\del p}-\bar{F}_{\inter,t}(u)\right)\;dt\\
&\geq& c_0\int_0^T|p|^2\;dt\,-\,\left(c_1+3\|\bar{G}_{\inter}\|_{C^0}\right)\cdot T,
\end{eqnarray*}
where we use that $F_{\prt,t}$ satisfies (F1) and $\left|d(\chi_{R_2}\circ\ln\circ |\cdot|)\cdot p\frac{\del}{\del p}\right|\leq 2$. Furthermore, we clearly crucially use that the $C^0$-norm of $\bar{G}_{\inter}$ (which equals the $C^0$-norm of $\tilde{G}_{\inter}$) is bounded. 
As in \cite{C} we see that the $L^2$-norm of $t\mapsto (q(t),p(t))$ is bounded. Since the gradient $\nabla\bar{F}_{\inter,t}$ is bounded with respect to the $C^0$-norm by \Cref{G1G2}, it further follows as in \cite[Section 5]{C} that there exists $c_3\geq 0$ such that $|\nabla\bar{F}_t(u)|\leq 2c_2|p|+c_3$ with $c_2\geq 0$ from (F2). As a consequence we get as in the proof of \cite[Lemma 5.3]{C} that $t\mapsto (q(t),p(t))$ is bounded  even with respect to the $H^1$-norm; using the Sobolev embedding theorem we can conclude that the same holds true for the $C^0$-norm. Since the $p$-component of every $T$-periodic solution of $\del_t u=J\nabla\bar{H}_t(u)$ and hence of $\del_t\bar{u}=J\nabla\bar{G}_t(\bar{u})$ of action $\leq a$ is hence uniformly bounded, the claim follows. \end{proof}

Summarizing, we find that for every given action bound $a\in\IR$ we can find $R_1,R_2>0$ such that $G_t=F_{\prt,t}+G_{\inter,t}$ and $\bar{G}_t=F_{\prt,t}+\bar{G}_{\inter,t}$ have the same set of $\phi$-periodic orbits of action less than or equal to $a$. 

\section{Floer curves in infinite dimensions}

In what follows we fix $R_1>0$ as in \Cref{bounded}, whereas $R_2>0$ will be determined below. As in finite dimensions, it follows that $\PP(\phi,\bar{G})$ with $\phi=\phi^A_T$ agrees with the set of critical points $\Crit(\AA_{\phi}^{\bar{G}})$ of the symplectic action functional $\AA=\AA_{\phi}^{\bar{G}}$ on $H^1_{\phi}(\IR,T^*Q\times\IH)$ for the time-dependent Hamiltonian $\bar{G}_t$, where $H^1_{\phi}(\IR,T^*Q\times\IH)$ contains all $H^1$-maps $\bar{u}:\IR\to T^*Q\times\IH$ with $\bar{u}(t+T)=\phi^A_{-T}\bar{u}(t)$. As in \cite{AS}, \cite{C}, \cite{Sch} and \cite{F}, \cite{FL} we follow the idea of A. Floer in \cite{Fl} to study flow lines of the gradient $\nabla\AA_{\phi}^{\bar{G}}$ of the action functional  with respect to the $L^2$-metric on $H^1_{\phi}(\IR,T^*Q\times\IH)$ given by the canonical Riemannian metric $\langle\cdot,\cdot\rangle$ on $\widetilde{M}=T^*Q\times\IH$. The reason why Floer preferred to choose the $L^2$-gradient over the more natural $H^1$-gradient follows from the observation that the gradient flow equation $\del_s\widetilde{u}=\nabla\AA_{\phi}^{\bar{G}}(\widetilde{u})$ for $\widetilde{u}:\IR\to H^1_{\phi}(\IR,T^*Q\times\IH)$ is equivalent to the perturbed Cauchy-Riemann equation $\CR_{\bar{G}}(\widetilde{u})=\del_s\widetilde{u}+J(\widetilde{u})\del_t\widetilde{u}+\nabla {\bar{G}}_t(\widetilde{u})=0$, where we now view $\widetilde{u}$ as a map from $\IR^2$ to $\widetilde{M}$ satisfying $\widetilde{u}(s,t+T)=\phi^A_{-T}\widetilde{u}(s,t)$ for all $(s,t)\in\IR^2$. Note that this is an infinite-dimensional analogue of the perturbed Cauchy-Riemann equation used to define Floer homology for general symplectomorphisms in \cite{DS}. The following main theorem of this paper is an analogue of \cite[Theorem 10.4]{FL}, see also \cite[Theorem 3.4]{F}.

\begin{theorem}\label{main}
Assume that there exist $\theta_1,\ldots,\theta_N\in H^{*\neq 0}(\Lambda^{\contr} Q,\IZ_2)$ with $\theta_1\cup\ldots\cup\theta_N\neq 0$. Then there exist $N$ maps $\widetilde{u}=\widetilde{u}_1,\ldots,\widetilde{u}_N:\IR^2\to T^*Q\times\IH_{\infty}\subset T^*Q\times\IH$ satisfying 
the Floer equation and $\phi_T^A$-periodicity condition
\begin{align*}
\del_s\widetilde{u}+J(\widetilde{u})\del_t\widetilde{u}+\nabla \bar{G}_t(\widetilde{u})=0,\qquad\widetilde{u}(s,t+T)=\phi_{-T}^A\widetilde{u}(s,t)
\end{align*}
with $\bar{G}_t(q,p,u_{\IH})=F_{\prt,t}(q,p)+\chi_{R_2}(\ln|p|)\cdot\chi_{R_1}(|u_{\IH}*\rho|_{C^3})\cdot G_{\inter,t}(q,u_{\IH})$ and $R_1>0$ as in \Cref{bounded}.
When $R_2>0$ is chosen large enough, for every $\alpha=1,\ldots,N$ the Floer curve $\widetilde{u}_{\alpha}$ connects two different solutions $\bar{u}=\bar{u}^-_{\alpha},\bar{u}^+_{\alpha}: \IR\to T^*Q\times\IH_{\infty}$ of 
\begin{align}
\label{ghameqn}
\del_t\bar{u}=X_t^G(\bar{u}),\qquad \bar{u}(t+T)=\phi^A_{-T}(\bar{u}(t))
\end{align}
with $G_t(q,p,u_{\IH})=F_{\prt,t}(q,p)+G_{\inter,t}(q,u_{\IH})$ in the sense that there exist sequences $s_{\alpha,n}^\pm\in\IR$ with $s_{\alpha,n}^\pm\to\pm\infty$ as $n\to\infty$ such that
\begin{align*}
%\label{asymptoticcondition}
\lim_{n\to\infty}\widetilde{u}_\alpha(s_{\alpha,n}^-,t)=\bar{u}^-_\alpha(t),\qquad\lim_{n\to\infty}\widetilde{u}_\alpha(s_{\alpha,n}^+,t)=\bar{u}^+_{\alpha}(t).
\end{align*}
Furthermore, since for the symplectic actions we have 
\begin{align*}
\mathcal{A}(\bar{u}^-_1)<\mathcal{A}(\bar{u}^+_1)\leq\mathcal{A}(\bar{u}^-_2)<\ldots<\mathcal{A}(\bar{u}^+_{N-1})\leq\mathcal{A}(\bar{u}^-_N)<\mathcal{A}(\bar{u}^+_N),    
\end{align*}
it follows that there are at least $N+1$ mutually different solutions of \eqref{ghameqn}.
\end{theorem}

For the proof we follow the strategy to combine the existence of Floer curves in cotangent bundles proven in \cite{C} with the infinite-dimensional Gromov-Floer compactness result from \cite{F}, \cite{FL}.\\

As a starting point we first approximate our infinite-dimensional Hamiltonian system by finite-dimensional ones. \\

For every $k\in\IN$ let $\IH^k$ denote the finite-dimensional subspace of $\IH$ which is spanned by all $e_n^{\pm}$ with $|n|=\sqrt{n_1^2+\ldots+n_d^2}\leq k$. First note that as in \cite{FL}, see also \cite[Prop. 2.1]{F}, the flow $\phi^A_t$ of the field Hamiltonian $H^A$ restricts to a unitary linear map on each $\IH^k$ since $\phi^A_t\cdot z_n=\exp(it\cdot\sqrt{n^2+1})\cdot z_n$. Here we use that $(e_n^{\pm})_{n\in\IZ^d}$ is a complete eigenbasis for $A$ with real eigenvalues. For every $k\in\IN$ let $\bar{G}^k_{\inter,t}=\bar{G}_{\inter,t}\circ\pi_k:T^*Q\times\IH\to\IR$ denote the Hamiltonian obtained by composing $\bar{G}_{\inter,t}$ with the projection $\pi_k: T^*Q\times\IH\to T^*Q\times\IH^k$. Furthermore note that for every fixed $t\in\IR$, $u\in T^*Q\times\IH$ the $\IH$-gradient $\nabla^{\IH}\bar{G}_{\inter,t}(u)$ can be expanded into a Fourier series, $$\nabla^{\IH}\bar{G}_{\inter,t}(u)=\sum_{n\in\IZ^d} \widehat{\nabla^{\IH}\bar{G}_{\inter,t}(u)}(n)\cdot z_n.$$ Then we have the following analogue of \cite[Lemmata 2.3 and 2.4]{F}, \cite[Lemma 3.2 and Lemma 6.1]{FL}  about finite-dimensional approximation.

\begin{lemma}\label{finite-dim} 
The gradients $\nabla\bar{G}^k_{\inter,t}$, $k\in\IN$ converge uniformly with their derivatives up to order $2$ to the gradient $\nabla\bar{G}_{\inter,t}$. Furthermore, for all $\delta\in\IN$ there exists $C_{\delta}>0$ such that $$|\widehat{\nabla^{\IH}\bar{G}_{\inter,t}(u)}(n)|\leq C_{\delta}\cdot |n|^{-\delta}$$ for all $t\in\IR$ and $u\in T^*Q\times\IH$. \end{lemma}

\begin{proof}
We start by observing that the statement holds if and only if it holds for $\bar{F}_{\inter,t}$, since $\phi^A_t$ is a unitary map. Note that $F^k_{\inter,t}=F_{\inter,t}\circ\pi_k$ is given by $F^k_{\inter,t}=f_t((u_{\IH}*\rho^k)(q))$ with $\rho^k=\sum_{|n|\leq k}\hat{\rho}(n)\cdot z_n\in\IH^k$. It follows that the $\IH$-gradient of $F^k_{\inter,t}$, $$\nabla^{\IH}F^k_{\inter,t}(u)=(B^{-1}\rho^k)(q-\cdot)\kolomtwee{\partial_1f_t((u_{\IH}*\rho^k)(q))}{\partial_2f_t((u_{\IH}*\rho^k)(q))}$$ converges to the $\IH$-gradient of $F_{\inter,t}$, $$\nabla^{\IH}F_{\inter,t}(u)=(B^{-1}\rho)(q-\cdot)\kolomtwee{\partial_1f_t((u_{\IH}*\rho)(q))}{\partial_2f_t((u_{\IH}*\rho)(q))}$$ uniformly with respect to $u\in T^*Q\times\IH$, using (F3). Passing from $F_{\inter,t}$ to $\bar{F}_{\inter,t}$, this result does not only continue to hold true, but even holds for all derivatives up to order $3$, since we may assume that we have a uniform bound on the $C^3$-norm on $u_{\IH}*\rho$ and hence in particular a uniform bound on $(u_{\IH}*\rho)(q)$. For the second statement we observe that since $\rho\in C^{\infty}(\IT^d,\IR)$ we know that the Fourier coefficients $\hat{\rho}(n)$ and hence also  $\widehat{\rho(\cdot-q)}(n)$ converge to zero with exponential speed. It immediately follows that the Fourier coefficients $\widehat{\nabla^{\IH}F_{\inter,t}(u)}(n)$ converge to zero with exponential speed and this convergence is uniform with respect to $t\in\IR$ and $u\in T^*Q\times\IH$ using (F3). \end{proof}

%$G^k_{\inter,t}=G_{\inter,t}\circ\pi_k$ is given by $G^k_{\inter,t}=f_t((u_{\IH}*\phi^A_{-t}\rho^k)(q))$ with $\rho^k=\sum_{|n|\leq k}\hat{\rho}(n)\cdot z_n\in\IH^k$. It follows that the $\IH$-gradient of $G^k_{\inter,t}$, $$\nabla^{\IH}G^k_{\inter,t}(u)=Df_t((u_{\IH}*\phi^A_{-t}\rho^k)(q))\cdot(B^{-1}\phi^A_{-t}\rho^k)(q-\cdot)$$ converges to the $\IH$-gradient of $G_{\inter,t}$, $$\nabla^{\IH}G_{\inter,t}(u)=Df_t((u_{\IH}*\phi^A_{-t}\rho)(q))\cdot(B^{-1}\phi^A_{-t}\rho)(q-\cdot)$$ uniformly with respect to $u\in T^*Q\times\IH$, using (F3). Passing from $G_{\inter,t}$ to $\bar{G}_{\inter,t}$, this result does not only continue to hold true, but even holds for the first and second derivatives, since we may assume that we have uniform bound on the $C^3$-norm on $u_{\IH}*\rho$ and hence in particular a uniform bound on $(u_{\IH}*\phi^A_{-t}\rho)(q)$. For the second statement we observe that $$\widehat{\nabla^{\IH}G_{\inter,t}(u)}(n)=Df_t((u_{\IH}*\phi^A_{-t}\rho)(q))\cdot\frac{\exp(-i\cdot T\cdot\sqrt{n^2+1})}{\sqrt{n^2+1}}\cdot \widehat{\rho(\cdot-q)}(n).$$ As in the proof of \Cref{bounded}, it follows from $\rho\in C^{\infty}(\IT^d,\IR)$ that the Fourier coefficients $\hat{\rho}(n)$ and hence also  $\widehat{\rho(\cdot-q)}(n)$ converge to zero with exponential speed. Note that this convergence is uniform with respect to $t\in\IR$ and $u\in T^*Q\times\IH$, since $Df_t$ is bounded by (F3). \end{proof}

After restricting to the finite-dimensional symplectic submanifold $T^*Q\times\IH^k$, note that, since $\bar{G}_{\inter,t}$ only has support in $\{u\in T^*Q\times\IH: |u_{\IH}*\rho|_{C^3}\leq R_1\}$, the finite-dimensional Hamiltonian $\bar{G}^k_{\inter,t}$ now has compact  support in $Q\times B_{R_k}(0)\subset Q\times\IH^k$, where $B_{R_k}(0)$ denotes a ball around $0$ in $\IH^k$. Note here we crucially need the assumption that all frequencies are present in $\rho$ in the sense that $\hat{\rho}(n)\neq 0$ for all $n\in\IZ^d$. \\

For fixed $a\in\IR$ let $\MM^{k,\leq a}_{\tau}$ denote the moduli space of Floer curves satisfying the $\tau$-dependent Floer equation with periodicity condition and with bounded symplectic action,
$$\MM^{k,\leq a}_{\tau}=\left\{\widetilde{u}=(\widetilde{u}_M,\widetilde{u}_{\IH}):\IR^2\to T^*Q\times\IH^k:\,\,(*1),(*2),(*3),(*4)\right\}$$ with 
\begin{eqnarray*}
&(*1):& \del_s\widetilde{u}+J(\widetilde{u})\del_t\widetilde{u}+\nabla F_{\prt,t}(\widetilde{u})+\sigma_\tau(s)\nabla \bar{G}^k_{\inter,t}(\widetilde{u})=0,\\
&(*2):& \widetilde{u}(s,t+T)=\phi_{-T}^A\widetilde{u}(s,t),\\
&(*3):& \pi_{\IH}\tilde{u}(s,\cdot)\to 0\,\,\textrm{as}\,\,s\to\pm\infty,\\
&(*4):& \AA^{F_{\prt}}_{\phi^A_T}(\tilde{u}(s,\cdot))\leq a\,\,\textrm{for}\,\,s\geq 2\tau+1,
\end{eqnarray*}
where $\sigma_{\tau}:\IR\to [0,1]$ is a smooth cut-off function with $\sigma_{\tau}(s)=0$ for $s\leq -1$, $s\geq 2\tau+1$ and $\sigma_{\tau}(s)=1$ for $0\leq s\leq 2\tau$. Note that we can define $N$ evaluation maps $\ev_{1,k,\tau},\ldots,\ev_{N,k,\tau}$ from $\MM^{k,\leq a}_{\tau}$ to the space $\Lambda^{\contr} Q=C^0_{\contr}(\IR/(T\IZ),Q)$ of contractible loops by 
$$\ev_{\alpha,k,\tau}:\MM^{k,\leq a}_{\tau}\to\Lambda^{\contr} Q,\,\,\widetilde{u}\mapsto \pi_Q\circ\widetilde{u}\left(2\tau\cdot\frac{\alpha}{N+1},\cdot\right),\,\,\alpha=1,\ldots,N,$$ where $\pi_Q$ denotes the projection from $T^*Q\times\IH^k$ onto the base manifold $Q$.\\

\begin{lemma}\label{bubbling-off} There is a uniform $C^1$-bound for the first derivative $T\widetilde{u}$ of  $\widetilde{u}\in\MM^{k,\leq a}_{\tau}$ which is independent of $k\in\IN$ and $\tau\geq 0$.\end{lemma}

\begin{proof} The first thing that needs to be observed is that the energy $E(\widetilde{u})$ of the Floer curves is uniformly bounded not just for all $\widetilde{u}\in\MM^{k,\leq a}_{\tau}$ and all $\tau\geq 0$, but also for all $k\in\IN$: Since $\|\bar{G}^k_{\inter,t}\|_{C^0}\leq \|\bar{G}_{\inter,t}\|_{C^0}$ for all $k\in\IN$, we get from \cite[Prop. 9.1.4]{MDSa} that $$E(\widetilde{u})\leq (a-b)+4T\|\bar{G}_{\inter,t}\|_{C^0},$$ where we again stress that $\bar{G}_{\inter,t}$ has finite $C^3$-norm by \Cref{G1G2}. Now let $\widetilde{u}^k$ be an arbitrary sequence of Floer curves in $\bigcup_k\bigcup_{\tau}\MM^{k,\leq a}_{\tau}$, where we want to assume without loss of generality that $\widetilde{u}^k\in\MM^{k,\leq a}_{\tau}$. As shown in \cite[Prop. 6.3]{FL}, see also \cite[Prop. 6.1]{F}, it follows precisely along the same lines as for sequences of Floer curves in fixed finite-dimensional exact symplectic manifolds using bubbling-off analysis and elliptic regularity that the $C^1$-norm of the first derivatives of the Floer curves $\widetilde{u}^k$ is uniformly bounded, $$\sup_k \|T\widetilde{u}^k\|_{C^1}<\infty.$$ 
We start by observing that the corresponding statement for the $C^0$-norm of $T\widetilde{u}$ in \cite[Lemma 6.2]{FL} is established using bubbling-off analysis as in \cite[Lemma 6.2]{F}: Indeed, assume without loss of generality that the first derivative is unbounded in the sense that
$$ C_k:=\max_{z=(s,t)\in\IR^2}\left\{|\partial_s \widetilde{u}^k(z)|\right\}=:\left|\partial_s\widetilde{u}^k(z_k)\right|\to\infty\,\,\textrm{as}\,\,k\to\infty.$$
Now consider the reparametrized map
$$\widetilde{v}^k:B_{\sqrt{C_k}}(0)\to T^*Q\times\IH^k:z\mapsto \widetilde{u}^k\left(\frac{z}{C_k}+z_k\right)$$
so that $|\partial_s\widetilde{v}^k(0)|=1$ and $|\partial_s\widetilde{v}^k(z)|\leq1$ for $|z|\leq\sqrt{C_k}$. By finiteness of area, it follows that for all $k$ there exists $\frac{\sqrt{C_k}}{2}\leq r_k\leq\sqrt{C_k}$ such that the length of the circle $\theta\mapsto\widetilde{v}^k(r_k e^{i\theta})$ goes to zero. By the exactness of $\omega$ the area of $\widetilde{v}^k_{r_k}$, which is the restriction of $\widetilde{v}^k$ to the disk of radius $r_k$, goes to zero. Finally one can employ an a priori bound to deduce that $\partial_s\widetilde{v}^k(0)\to 0$ as $k\to\infty$, and the proof of $\sup_k \|T\widetilde{u}^k\|_{C^0}<\infty$ follows by contradiction. For elliptic bootstrapping we can use \cite[App. B]{MDSa} to deduce that $$\|\nabla G^k(\tilde{u}^k)\|_{H^{2,p}}\leq c_1\cdot \|\nabla G^k\|_{C^2}\cdot(1+\|\tilde{u}^k\|_{L^\infty})(1+\|\tilde{u}^k\|_{H^{2,p}})$$ and hence the $H^{3,p}$-norm of $\widetilde{u}$ is bounded, see also the proof of \cite[Prop. 6.1]{F}. As mentioned already after the proof of \Cref{bounded}, replacing the $C^3$-norm in the definition of $\bar{G}_{\inter,t}$ by the $C^{\kappa}$-norm, we can obtain $H^{\kappa,p}$-bounds for any $\kappa$. \end{proof}

From now on fix $\theta_1,\ldots,\theta_N\in H^{*\neq 0}(\Lambda^{\contr} Q,\IZ_2)$ with $\theta_1\cup\ldots\cup\theta_N\neq 0$.

\begin{proposition}\label{cuplength}
There exists $a\in\IR$ such that for every $k\in\IN$ and $\tau\geq 0$ we have $$\ev_{1,k,\tau}^*\theta_1\cup\ldots\cup\ev_{N,k,\tau}^*\theta_N\neq 0\in H^*(\MM^{k,\leq a}_{\tau},\IZ_2).$$
\end{proposition}

\begin{proof}
The key ingredient in the proof is \cite[Theorem 7.6]{C}. Note that for $\tau=0$ the corresponding moduli spaces $\MM^{k,\leq a}_0$ consist of Floer curves $\widetilde{u}:\IR^2\to T^*Q\times\IH^k$ satisfying the Floer equation $\del_s\widetilde{u}+J(\widetilde{u})\del_t\widetilde{u}+\nabla F_{\prt,t}(\widetilde{u})=0$. Writing $\widetilde{u}$ as a pair $\widetilde{u}=(\widetilde{u}_M,\widetilde{u}_{\IH})$ with $\widetilde{u}_M:\IR^2\to T^*Q$ and $\widetilde{u}_{\IH}:\IR^2\to\IH^k$, we find that the Floer equation (as well as the periodicity condition) decouples,
\begin{eqnarray*}
\overline{\partial}_{J_M}\widetilde{u}_M+\nabla F_{\prt,t}(\widetilde{u}_M)=0,&&\widetilde{u}_M(s,t+T)=\widetilde{u}_M(s,t),\\
\CR\widetilde{u}_{\IH}=0,&&\widetilde{u}_{\IH}(s,t+T)=\phi^A_{-T}\widetilde{u}_{\IH}(s,t).
\end{eqnarray*}
Since we assume that $\widetilde{u}_{\IH}(s,\cdot)\to 0$ as $s\to\pm\infty$, it follows that $\widetilde{u}_{\IH}\equiv 0$. But this implies that $\MM^{k,\leq a}_0$ precisely agrees with the moduli space of Floer curves for which Cieliebak stated his Theorem 7.6 in \cite{C}. Indeed it follows from \cite[Theorem 7.6]{C} that there exists $a\in\IR$ such that  $$\ev_{1,k,0}^*\theta_1\cup\ldots\cup\ev_{N,k,0}^*\theta_N=\ev_{k,0}^*(\theta_1\cup\ldots\cup\theta_N)\neq 0\in H^*(\MM^{k,\leq a}_0,\IZ_2),$$ where $\ev_{1,k,0}=\ldots=\ev_{N,k,0}=:\ev_{k,0}$ for all $k\in\IN$ by definition. 

The corresponding statement for all $\tau\geq 0$ now follows from a standard finite-dimensional homotopy argument, where the Gromov-Floer compactness theorem plays the central role. We start with establishing the required $C^0$-bounds for Floer curves in $T^*Q\times\IH^k$. First, since $\bar{G}_{\inter,t}(q,p,u_{\IH})=\chi_{R_2}(\ln|p|)\cdot\chi_{R_1}(|u_{\IH}*\rho|_{C^3})\cdot G_{\inter,t}(q,u_{\IH})=0$ and hence $\bar{G}_t(q,p,u_{\IH})=F_{\prt,t}(q,p)$ for $|p|$ sufficiently large, it follows that the $C^0$-bound for Floer curves in $T^*Q$ established in \cite{C} immediately establishes a $C^0$-bound for our Floer curves in $T^*Q\times\IH^k$ which is even independent of $k\in\IN$. On the other hand, the compactness problems due to the unboundedness of $\IH^k$ is taken care of as in \cite[Prop. 7.2]{FL}: Since $\bar{G}^k_{\inter,t}$ has bounded support in $Q\times B_{R_k}(0)\subset Q\times\IH^k$ and $F_{\prt,t}$ by choice does not depend on the $\IH^k$-coordinates, it follows from the maximum principle for unperturbed holomorphic curves in $\IH^k$ that the Floer curve has to stay inside the same bounded subset. 

After establishing the necessary $C^0$-bounds, we observe that by \Cref{bubbling-off} we know that we even have uniform bounds for the $C^2$-norm of $\widetilde{u}\in\MM^{k,\leq a}$. By the classical elliptic bootstrapping arguments it follows that we hence have compactness with respect to the $C^1$-norm modulo breaking of Floer curves. \end{proof}

For every $k\in\IN$ let $\widetilde{u}^k\in\MM^{k,\leq a}_k$ be an arbitrary Floer curve in $T^*Q\times\IH^k$ for $\tau=k$. As in the proof of \cite[Theorem 10.4]{FL} the idea is to apply the infinite-dimensional generalization of the Gromov-Floer compactness result from \cite{F} and \cite{FL} to sequences of $N$ shifted Floer curves $\widetilde{u}^k_{\alpha}$, $\alpha=1,\ldots,N$ in order to obtain Floer curves in the infinite-dimensional symplectic manifold  $\widetilde{M}=T^*Q\times\IH$. More precisely, combining the proof of \cite[Theorem 10.4]{FL}, which itself essentially is based on \cite[Lemma 8.1]{F}, \cite[Theorem 8.1]{FL}, with the $C^0$-bounds from \cite[Theorem 5.4]{C} leads to a proof of the following

\begin{proposition}\label{compactness}
For every $\alpha=1,\ldots,N$, a subsequence of the sequence of \emph{shifted} Floer curves $$\widetilde{u}^k_{\alpha}:\IR^2\to T^*Q\times\IH^k,\,\, \widetilde{u}^k_{\alpha}(s,t)= \widetilde{u}^k(s+2k\frac{\alpha}{N+1},t)$$ $C^1$-converges to a solution $\widetilde{u}=\widetilde{u}_{\alpha}:\IR^2\to T^*Q\times\IH$ of the Floer equation 
$$\del_s\widetilde{u}+J(\widetilde{u})\del_t\widetilde{u}+\nabla \bar{G}_t(\widetilde{u})=0,\,\,\widetilde{u}(s,t+T)=\phi_{-T}^A\widetilde{u}(s,t),$$ which satisfies the following asymptotic conditions: there exists sequences $s_{\alpha,n}^\pm\in\IR$ with $s_{\alpha,n}^\pm\to\pm\infty$ as $n\to\infty$ such that
\begin{align*}
\lim_{n\to\infty}\widetilde{u}_{\alpha}(s_{\alpha,n}^-,t)=\bar{u}^-_{\alpha}(t),\qquad\lim_{n\to\infty}\widetilde{u}(s_{\alpha,n}^+,t)=\bar{u}^+_{\alpha}(t)
\end{align*}
in the $C^1$-sense where $\bar{u}^-_{\alpha}$ and $\bar{u}^+_{\alpha}$ are two $\phi_T^A$-periodic orbits of $\bar{G}_t$.
\end{proposition}

\begin{proof} Note that we have $\sigma_k(s+2k\frac{\alpha}{N+1})\to 1$ for every $(s,t)\in\IR^2$ as $k\to\infty$. 
While \Cref{bubbling-off} is known to be the main ingredient for Gromov-Floer compactness in the case of closed finite-dimensional symplectic manifolds, here we need to deal with the noncompactness of the target manifold. While we have already discussed how $C^0$-bounds can be established, we now turn to the most striking problem, namely the problem that the target manifold $\widetilde{M}=T^*Q\times\IH$ is indeed infinite-dimensional. As in \cite{F} and \cite{FL} we start by observing that we can write the finite-dimensional Floer curve as a tuple
$$\widetilde{u}^k=(\widetilde{u}^{k,\ell},\widetilde{u}^{k,\ell}_\perp):\IR^2\to (T^*Q\times\IH^{\ell})\times(\IH^k/\IH^{\ell})=T^*Q\times\IH^k\subset T^*Q\times\IH,$$
where $\widetilde{u}^{k,\ell}_\perp$ denotes the normal component of $\widetilde{u}^k$. As in \cite[Prop. 7.1]{F}, \cite[Prop. 7.2]{FL} we prove in \Cref{small-divisors} below that we have $$\sup_{k\geq\ell}\left\|\widetilde{u}^{k,\ell}_\perp\right\|_{C^1}\to0\quad\mathrm{as}\quad \ell\to\infty.$$
On the other hand, as in the proof of \cite[Lemma 8.1]{F}, \cite[Theorem 8.1]{FL} it follows from \Cref{bubbling-off} using the standard elliptic bootstrapping argument together with a diagonal subsequence argument that for each $\alpha=1,\ldots,N$ there is a subsequence of $(\widetilde{u}^k_{\alpha})_k$ such that for all $\ell\in\IN$ the corresponding sequence  $(\widetilde{u}^{k,\ell}_{\alpha})_k$ $C^1$-converges to a map $\widetilde{u}^\ell_{\alpha}:\IR^2\to T^*Q\times \IH^{\ell}$ as $k\to\infty$. Together with \Cref{small-divisors} about the normal component this proves that the given subsequence of $(\widetilde{u}^k_{\alpha})_k$ converges in the $C^1$-sense to a map $\widetilde{u}_{\alpha}$ which solves the Floer equation for $\bar{G}_t=F_{\prt,t}+\bar{G}_{\inter,t}$ using again \Cref{finite-dim}. 

The asymptotic condition is proven as in the proof of \cite[Theorem 5.1]{F}, \cite[Theorem 8.2]{FL} and hence again crucially relies on \Cref{small-divisors}. For this choose sequences $s_{n,\alpha}^\pm\in\IR$ with $n\leq s_{n,\alpha}^+\leq 2n$ and $n\leq -s_{n,\alpha}^-\leq 2n$ such that 
$$
\int_0^T\left|\del_s\widetilde{u}_{\alpha}(s_{n,\alpha}^\pm,t)\right|^2dt\leq \frac{E(\widetilde{u}_{\alpha})}{n}\to 0\text{ as } n\to\infty,
$$ 
where we use that the energy of the limiting Floer curve $\widetilde{u}_{\alpha}$ is still bounded by $(a-b)+4T||\bar{G}^k_{\inter,t}||_{C^0}$. Now we write $\widetilde{u}_{\alpha}=(\widetilde{u}_{\alpha}^{\ell},\widetilde{u}_{\alpha,\perp}^{\ell}):\IR^2\to (T^*Q\times\IH^{\ell})\times \IH/\IH^{\ell}$ for $\ell\in\IN$. Using the $C^0$-bounds and after passing to a diagonal subsequence we can assume that $\widetilde{u}_{\alpha}^{\ell}(s_{n,\alpha}^\pm,\cdot)$ $C^1$-converges as $n\to\infty$ for all $\ell$ simultaneously. Using again \Cref{small-divisors} we can deduce that $\widetilde{u}_{\alpha}(s_{n,\alpha}^\pm,\cdot)$ $C^1$-converges as $n\to\infty$ to $\phi^A_T$-periodic orbits $\bar{u}_{\alpha}^\pm$ of $\bar{G}_t(p,q,u_{\IH})=F_{\prt,t}(q,p)+\chi_{R_2}(\ln|p|)\cdot\chi_{R_1}(|u_{\IH}*\rho|_{C^3})\cdot G_{\inter,t}(q,u_{\IH})$. \\

The following lemma is just a recast of \cite[Prop. 7.1]{F}, \cite[Prop. 7.2]{FL} and added for completeness. 

\begin{lemma} \label{small-divisors}
We have $$\sup_{k\geq\ell}\left\|\widetilde{u}^{k,\ell}_\perp\right\|_{C^1}\to0\quad\mathrm{as}\quad \ell\to\infty.$$
\end{lemma}

\begin{proof}
Since each $\tilde{u}^k:\IR^2\to T^*Q\times\IH^k$ satisfies the Floer equation $\del_s\widetilde{u}^k+J(\widetilde{u}^k)\del_t\widetilde{u}^k+\nabla \bar{G}^k_{s,t}(\widetilde{u})=0$, it follows that the $\IH$-component $\widetilde{u}^k_{\IH}:=\pi_{\IH}\widetilde{u}^k:\IR^2\to\IH^k$ satisfies the equation $\del_s\widetilde{u}^k_{\IH}+i\del_t\widetilde{u}^k_{\IH}+\nabla^{\IH} \bar{G}^k_{s,t}(\widetilde{u}^k)=0$. Note that here $\nabla^{\IH} \bar{G}^k_{s,t}$ denotes the (orthogonal) projection of the gradient $\nabla \bar{G}^k_{s,t}$ onto $\IH^k\subset\IH$ and we identify $\IH$ with the subspace of $\IH\otimes\IC$ on which $J_{\IH}=i$. In what follows we view each $\IH$-component $\widetilde{u}^k_{\IH}$ as a map from $\IR$ into $L^2_{\phi}(\IR,\IH)$, where the Hilbert space $L^2_{\phi}(\IR,\IH)$ consists of all maps $\bar{u}\in L^2(\IR,\IH)$ satisfying the periodicity condition $\bar{u}(t+T)=\phi^A_{-T}\bar{u}(t)$. Setting $\nabla^{\IH} \bar{G}^k(\widetilde{u}^k)(s)(t):=\nabla^{\IH} \bar{G}^k_{s,t}(\widetilde{u}^k(s,t))$, we can view $\nabla^{\IH} \bar{G}^k(\widetilde{u}^k)$ also as a map from $\IR$ into $L^2_{\phi}(\IR,\IH)$. The operator $-i\partial_t$ on $L^2_{\phi}(\IR,\IH)$ has a complete basis of eigenfunctions $u_{m,n}(t)=e^{\lambda_{m,n}it}\cdot z_n$ with corresponding eigenvalues $\lambda_{m,n}=m\frac{2\pi}{T}-\sqrt{n^2+1}$ for $m\in\IZ$ and $n\in\IZ^d$. Now we apply the corresponding Fourier transform to $\widetilde{u}^k_{\IH}:\IR\to L^2_{\phi}(\IR,\IH)$ and $\nabla^{\IH} \bar{G}^k(\widetilde{u}^k):\IR\to L^2_{\phi}(\IR,\IH)$ to obtain sequences of maps $w^k_{m,n}:=\widehat{\widetilde{u}^k_{\IH}}(m,n),g^k_{m,n}:=\widehat{\nabla^{\IH} \bar{G}^k(\widetilde{u}^k)}(m,n):\IR\to\IC$ satisfying $$(w^k_{m,n})'(s)=\lambda_{m,n}w^k_{m,n}(s)+g^k_{m,n}(s)\,\,\textrm{and}\,\,w^k_{m,p}(s)\to 0\,\,\textrm{as}\,\,s\to\pm\infty.$$ Since $\bar{G}_t$ is smooth with respect to $t$ and the $C^2$-norms of $\widetilde{u}^k$ are uniformly bounded in $k\in\IN$, it follows from \Cref{finite-dim} that for all $\delta\in\IN$ there exists $C_{\delta}>0$ such that $$|g_{m,n}^k(s)|\leq C_{\delta}\cdot |m|^{-2}\cdot |n|^{-\delta}$$ independent of $k\in\IN$, $s\in\IR$. On the other hand we know from the proof of \Cref{bounded} that there exist $c>0$ and $r>0$ such that $$|\lambda_{m,n}|=\frac{2\pi}{T}\cdot \sqrt{n^2+1}\cdot \Big|\frac{T}{2\pi}-\frac{m}{\sqrt{n^2+1}}\Big|\geq c\cdot |n|^{-r}.$$ Using the elementary estimate $\|w^k_{m,n}\|_{C^0}\leq \|g^k_{m,n}\|_{C^0}/|\lambda_{m,n}|$ from \cite[Lemma 7.2]{F}, \cite[Lemma 7.1]{FL} we get that  $$|\widehat{\widetilde{u}^k(s)}(m,n)|=|w_{m,n}^k(s)|\leq C_{\delta}/c\cdot |m|^{-2}\cdot |n|^{-\delta+r}$$ for all $\delta\in\IN$ independent of $k\in\IN$, $s\in\IR$. From $$|\partial_t^j\widetilde{u}^{k,\ell}_\perp(s,t)|^2\leq\sum_{|n|=\ell+1}^\infty \left(\sum_{m\in\IZ}|\widehat{\widetilde{u}^k(s)}(m,n)||m|^j\right)^2,$$ we can conclude that $$\sup_{k\geq\ell}\left\|\widetilde{u}^{k,\ell}_\perp\right\|_{C^0},\,\sup_{k\geq\ell}\left\|\del_t\widetilde{u}^{k,\ell}_\perp\right\|_{C^0}\to0\quad\mathrm{as}\quad \ell\to\infty.$$ On the other hand, using the Floer equation $\del_s\widetilde{u}^k_{\IH}+i\del_t\widetilde{u}^k_{\IH}+\nabla^{\IH} \bar{G}^k_{s,t}(\widetilde{u}^k)=0$ and \Cref{finite-dim}, we also get $\displaystyle\sup_{k\geq\ell}\left\|\del_s\widetilde{u}^{k,\ell}_\perp\right\|_{C^0}\to 0$ as $\ell\to\infty$, and the claim follows. \end{proof}

In order to see that we actually obtain $\phi^A_T$-periodic orbits of the original Hamiltonian $G_t(q,p,u_{\IH})=F_{\prt,t}(q,p)+G_{\inter,t}(q,u_{\IH})$ as long as $R_2>0$ is chosen sufficiently large, note that the $\phi^A_T$-periodic orbits $\bar{u}_{\alpha}^\pm$ of $\bar{G}_t$ that we have found have action less than or equal to $a+4T\|\bar{G}_{\inter,t}\|_{C^0}$, see \cite[Theorem 9.1.13]{MDSa} for a similar result. 
Now use \Cref{p-cutoff} to find $R_2>0$ such that the $\phi^A_T$-periodic orbits of $G_t=F_{\prt,t}+G_{\inter,t}$ agree with those of $\bar{G}_t=F_{\prt,t}+\bar{G}_{\inter,t}$ with $\bar{G}_{\inter,t}(q,p,u_{\IH})=\chi_{R_2}(\ln|p|)\cdot\chi_{R_1}(|u_{\IH}*\rho|_{C^3})\cdot G_{\inter,t}(q,u_{\IH})$ as long as the symplectic action is less than or equal to $a+4T\|\bar{G}_{\inter,t}\|_{C^0}$. On the other hand, working with the modified Hamiltonian $\bar{G}_t$ instead of $G_t$ allowed us to employ a maximum principle as well as the $C^0$-bounds for the $p$-component of Floer curves established in \cite{C}.\par

Next, in order to see that we obtain mutually different periodic solutions, we follow ideas from \cite{Sch}. Since $\ev_{1,k,\tau}^*\theta_1\cup\ldots\cup\ev_{N,k,\tau}^*\theta_N\neq 0\in H^*(\MM^{k,\leq a}_k,\IZ_2)$, we claim that, possibly after choosing a different sequence $\widetilde{u}^k\in\MM^{k,\leq a}_k$, we have that \begin{align*}
\mathcal{A}(\bar{u}^+_{\alpha})-\mathcal{A}(\bar{u}^-_{\alpha})=E(\widetilde{u}_{\alpha})=\int_{-\infty}^{+\infty}\int_0^T\left|\partial_s\widetilde{u}_{\alpha}\right|^2dt\;ds>0,
\end{align*} for all $\alpha=1,\ldots,N$ which in turn implies that $\bar{u}^-_{\alpha}$ and $\bar{u}^+_{\alpha}$ need to be pairwise different for all $\alpha=1,\ldots,N$. Note that, if this is not the case, then we find that in the limit $\tau\to\infty$ the image of the evaluation map in $(\Lambda^{\contr} Q)^N$ is degenerate in the sense that it factors through the map $\bigcup_{i=1}^N (\Lambda^{\contr}Q)^{i-1}\times\{\textrm{point}\}\times (\Lambda^{\contr}Q)^{N-i-1}\to (\Lambda^{\contr}Q)^N$. But this contradicts the fact that we have chosen $a\in\IR$ as in \Cref{cuplength}. Since furthermore $\AA(\bar{u}^{\pm}_{\alpha})\leq\AA(\bar{u}^{\pm}_{\beta})$ if $\alpha<\beta$ since the Floer equation is a gradient flow equation for the symplectic action functional, this proves the existence of $N+1$ mutually different solutions. \par

Further, as described in \cite[Section 8]{C} it follows from work of D. Sullivan that the finiteness of the fundamental group of $Q$ implies that the $\IZ_2$-cuplength of the space of contractible loops in $Q$ is infinite. \par 

And finally, analogous to the proof of \cite[Prop. 3.5]{F}, \cite[Theorem 8.3]{FL} we note that the limit Floer curve $\widetilde{u}$ has image in $T^*Q\times\IH_{\infty}$, since for the Fourier coefficients of the Floer curve $\widetilde{u}$ we know that for every $\delta\in\IN$ there exists $C_{\delta}>0$ such that $$|\widehat{\widetilde{u}(s)}(m,n)|\leq C_{\delta}/c \cdot |m|^{-2 +\frac{1}{2}}\cdot |n|^{-\delta+r}.$$ \end{proof}

\end{document}